\documentclass{article}
\usepackage[utf8]{inputenc}
\usepackage{amsmath, amsfonts, amssymb, amsthm, framed}
\usepackage[top=1in, bottom=1in]{geometry}

\newtheorem{theorem}{Theorem}
\newtheorem{lemma}[theorem]{Lemma}
\newtheorem{proposition}[theorem]{Proposition}
\newtheorem{corollary}[theorem]{Corollary}

\newtheorem{conjecture}[theorem]{Conjecture}
\newtheorem{question}[theorem]{Question}

\newcommand{\red}{\texttt{R}}
\newcommand{\blue}{\texttt{B}}
\newcommand{\tie}{\texttt{T}}

\usepackage{color}

\newcommand{\cho}[2]{\begin{pmatrix}
#1\\#2
\end{pmatrix}}

\newcommand{\ignore}[1]{}

\title{\vspace*{-.5in}Central Limit Theorem for Majority Dynamics: Bribing Three Voters Suffices}
\author{Ross Berkowitz\footnote{Dept.\ of Mathematics, Yale University, New Haven CT, USA \qquad \texttt{ross.berkowitz@yale.edu}} \and Pat Devlin\footnote{Dept.\ of Mathematics, Yale University, New Haven CT, USA \qquad \texttt{patrick.devlin@yale.edu}}}

\date{October 15, 2020}

\begin{document}

\maketitle
\renewcommand{\thefootnote}{\fnsymbol{footnote}}

\setcounter{footnote}{0}
\renewcommand{\thefootnote}{\arabic{footnote}}

\abstract{Given a graph $G$ and some initial labelling $\sigma : V(G) \to \{Red, Blue\}$ of its vertices, the \textit{majority dynamics model} is the deterministic process where at each stage, every vertex simultaneously replaces its label with the majority label among its neighbors (remaining unchanged in the case of a tie).  We prove---for a wide range of parameters---that if an initial assignment is fixed and we independently sample an Erd\H{o}s--R\'enyi random graph, $G_{n,p}$, then after one step of majority dynamics, the number of vertices of each label follows a central limit law.  As a corollary, we provide a strengthening of a theorem of Benjamini, Chan, O'Donnell, Tamuz, and Tan about the number of steps required for the process to reach unanimity when the initial assignment is also chosen randomly.

Moreover, suppose there are initially three more red vertices than blue.  In this setting, we prove that if we independently sample the graph $G_{n,1/2}$, then with probability at least $51\%$, the majority dynamics process will converge to every vertex being red.  This improves a result of Tran and Vu who addressed the case that the initial lead is at least 10.}


\section{Introduction}
\textit{Majority dynamics} is a model for belief propagation on a graph.  In this process, each vertex of the graph has one of two opinions (which we will refer to as the colors `red' and `blue').  Each round, the opinions of the vertices are all updated simultaneously, and each vertex replaces their opinion so as to match the majority opinion that was held by their neighbors  (in the case of a tie, a vertex does not change its opinion).  This updating step is then iterated, and the resulting process is referred to as majority dynamics.  This model has been used in a variety of settings including social choice theory \cite{social1, nguyen2020}, economics \cite{econ2, econ1}, and statistical physics \cite{physicis1, physics2}, and many related processes have been proposed as well (see \cite{abdullah2015, amir2019majority, balogh2007bootstrap, zehmakan2018two} for just a few examples).

The majority dynamics model has gotten a lot of attention for fixed deterministic graphs, with many of the primary questions being related to the long-term behavior \cite{periodic, ginosar2000}.  The problem has also been studied in the settings of fixed graphs with random initial colorings \cite{randomColors, tamuz2015, colorWar}.  We direct the interested reader to \cite{mossel2014} and \cite{shah2009gossip} for helpful surveys.

The situation of majority dynamics on random graphs has also been a focus of interest, but as discussed below this setting is not particularly well understood.  Before diving into this literature, let us first establish the following notation to be used throughout.

\subsubsection*{Notation}
The Erd\H{o}s--R\'enyi random graph, $G_{n,p}$, is defined as the graph with a vertex set of size $n$ and each potential edge independently included with probability $p$.  Given a graph and an intial coloring $R_0$ and $B_0$, we will let $R_i$ (resp.\ $B_i$) denote the set of vertices that are red (resp.\ blue) after $i$ steps of majority dynamics.  The initial coloring $R_0, B_0$ may be fixed or determined randomly, but in either case it will always be independent of the random graph $G_{n,p}$.

Asymptotic notation is understood to mean as the relevant parameter (typically $n$) tends to infinity.  We say an event happens \textit{with high probability} to mean its probability tends to $1$. All logarithms have base $e$.  Finally, we use the standard notation $\Phi(t) = (2 \pi)^{-1/2} \int_{-\infty} ^{t} e^{-x^2/2} dx$ to denote the probability that a mean 0, variance 1 Gaussian random variable is at most $t$.

\subsubsection*{Previous results for random graphs}
In studying majority dynamics on random graphs, a common approach is as follows.  One first tries to control the number of red vertices after one \cite{benjamini2016, zehmakan, tran-vu} or two \cite{whpFolks, zehmakan} steps (hoping that in these initial steps, one color will have managed to establish a substantial lead).  After this, one typically argues that with probability tending to 1, if any color ever manages to establish a sufficiently large lead, then this lead is going to increase \textit{even if the vertices were adversarially recolored} (provided this recoloring preserves the number of vertices of each color).

The second part of this strategy (i.e., showing that with high probability large leads cannot be squandered) often involves some variant of a relatively direct union bound using the fact that the random graph is sufficiently expansive.  On the other hand, the first part of this strategy (i.e., controlling the number of vertices of each color after one or two steps) is considerably more nuanced.

Following this general strategy, Benjamini, Chan, O'Donnell, Tamuz, and Tan \cite{benjamini2016} obtained rough estimates for the mean and variance of $|R_1|$ under the assumption that the initial coloring $R_0, B_0$ is assigned uniformly at random with $|R_0| \geq |B_0|$.  Their argument was based on a combination of Fourier analysis and spectral graph techniques, and using their estimates, they were able to prove that if $p \geq C/\sqrt{n}$ and $|R_0| \geq |B_0|$, then $\mathbb{P}(R_4 = V) \geq 0.4 - o(1)$.  Thus, for $p \geq C /\sqrt{n}$, they proved that with probability at least $0.4$, for a randomized initial coloring of $G_{n,p}$, the color that was initially ahead will propagate to include all the vertices of the graph within at most four steps.  When $p \geq C/n^{1/3}$, their same estimates on $|R_1|$ also prove that with probability at least $0.4$, the same result holds after at most 3 steps.

Analyzing the same setting, Fountoulakis, Kang, and Makai \cite{whpFolks} improved on the first of these results by focusing instead on estimating the mean and variance of $|R_2|$.  Namely, they proved that if $R_0, B_0$ is chosen uniformly at random and $p \geq C/\sqrt{n}$, then \textit{with probability tending to 1}, the color that was initially ahead will propagate to the entire graph within at most four steps.

For smaller values of $p$, Zehmakan \cite{zehmakan} studied the case that the vertices of $G_{n,p}$ are each initially colored blue independently with probability $q$.  For $q \leq 1/2 - \omega(1/\sqrt{np})$ he proved that when $p \geq (1+\varepsilon) \log(n)/n$, then with high probability, majority dynamics converges to all vertices being red after a bounded number of steps.  (For smaller $p$, there will be isolated blue vertices.)  G\"artner and Zehmakan \cite{gartner2018majority} also proved a similar result for $d$-regular random graphs.  See \cite{benjamini2016} for several conjectures about how the majority dynamics process evolves for values of $p$ less than $C/\sqrt{n}$ and $|R_0| \sim |B_0|$.

Finally, in a different direction, Tran and Vu \cite{tran-vu} considered the case that the initial coloring is \textit{fixed} with $|R_0| - |B_0| \geq 10$.  In this setting, they argued that with probability at least $90\%$, the graph $G_{n,1/2}$  will result in all of its vertices being red.

\subsubsection*{Our results}
It is worth emphasizing that in the above papers, the key first step has been to analyze the distribution of $|R_1|$ or $|R_2|$.  In fact, essentially all of these previous results have followed from some new estimate on the mean or variance of these random variables.  To quote \cite{benjamini2016}, ``the main task is to analyze what happens at time 1."  This brings us to our primary result.

\begin{theorem}\label{good CLT}
Let $R_0$, $B_0$ be any fixed initial coloring of $n$ vertices, and suppose we independently sample the random graph $G_{n,p}$.  Let $|R_1|$ denote the number of vertices that will be red after one step of majority dynamics.  Let $\Delta = \Big| |R_0| - |B_0| \Big|$, let $\sigma = \sqrt{np(1-p)}$, and define
\[
\mu = \mathbb{P} \Big(Bin(|R_0|,p) \geq Bin(|B_0|,p) \Big) \cdot \mathbb{P} \Big(Bin(|R_0|,p) \leq Bin(|B_0|,p) \Big).
\]
\begin{itemize}
\item[(i)] For any choice of parameters, the variance of $|R_1|$ satisfies $Var(|R_1|) = n \mu + \mathcal{O}(\Delta + n/\sigma)$.
\item[(ii)] If $\log(1/\mu) = o(\log(\sigma))$ and $\sigma \to \infty$, then $Var(|R_1|) \sim n \mu$, and $|R_1|$ obeys a central limit law.  Namely, for all fixed reals $a < b$, we have
\[
\lim_{n \to \infty} \mathbb{P} \left( \dfrac{|R_1| - \mathbb{E}|R_1|}{\sqrt{n \mu}} \in [a,b] \right) = \dfrac{1}{\sqrt{2\pi}} \int_{a} ^{b} e^{-x^2 /2} dx = \Phi(b) - \Phi(a).
\]
\end{itemize}
In particular, if $\Delta p = \mathcal{O}(\sigma)$ and $\sigma \to \infty$, then the hypotheses and conclusion of (ii) hold.
\end{theorem}
This provides an unqualified upper bound that $Var(|R_1|) = \mathcal{O}(n)$, and for a wide range of parameters, it also gives a central limit theorem.  Crucially, the statement is for a given value of $|R_0|$ (which may certainly depend on $n$) as opposed to $|R_0|$ being a random variable.  Conditioning on $|R_0|$ in this way enables us to control all the moments of $|R_1|$, whereas otherwise the variance in this quantity would be dominated by the variance of $|R_0|$.

Theorem \ref{good CLT} is proven by the method of moments, but to illustrate the main challenge, consider the calculation of the variance of $|R_1|$.  The difficulty here is that one quickly arrives at this variance as a sum of covariances, but each covariance term is large in absolute value.  Bounding this sum by its maximum term gives the incorrect asymptotic growth, and the key idea is therefore devising a way to analyze the large amount of cancellation in these sums.  We do this by studying the relevant $p$-biased Fourier coefficients (with random variables viewed as functions of the edges present in $G_{n,p}$).

At this point, it is worth noting that for a wide range of parameters (concretely, for $|R_0| = |B_0|$ and $p$ constant) the distribution of $|R_1|$ converges to a normal, but the distribution of $|R_2|$ certainly will not.  This is simply because if $|R_1|$ is sufficiently far from $n/2$ (e.g., at least $\Omega(\sqrt{n/p})$ away [see Proposition \ref{lemma:sex panther}]) then this would force $|R_2|$ to be very far away from $n/2$.  Thus, the mass of the distribution of $|R_1|$ is pushed very far away from $n/2$, resulting in a particularly bimodal distribution (having mean $n/2$, variance $\Omega(n^2)$, but being bounded between $0$ and $n$).  The distribution of $|R_3|$ would be even more exaggerated, and it would likely be essentially two peaks at $0$ and $n$.  For many settings of parameters, we believe the distribution of $|R_4|$ would converge to literally two point masses at $0$ and $n$.  Thus, we ask the following (with the unproven belief that $|R_k|$ is often only interesting for $k \in \{1,  2\}$).

\begin{question}
How is $|R_k|$ distributed?  (For concreteness, take $|R_0| = |B_0|$ and $p = 1/2$)
\end{question}

Regardless of the above, Theorem \ref{good CLT} is already strong enough to imply most of the known results (including the bounds on $|R_2|$ from \cite{whpFolks}), but to leverage this we use the following.

\begin{proposition}\label{lemma:sex panther}
For each $\varepsilon > 0$, there are constants $C>0$ and $q \in (0,1)$ such that if $np(1-p) > C$ and $G \sim G_{n,p}$, then with probability at least $1 - q^n$, the following statement simultaneously holds for every choice of $\delta \in (\varepsilon, 1-\varepsilon)$ and every $R \subseteq V(G)$.  If $|R| \geq n/2 + \alpha \sqrt{n(1-p)/p}$ and 
\[
 \sup_{0 \leq \gamma \leq 2} e^{-\gamma^2 / 2} \left[\Phi \left(\frac{\gamma - 2\alpha}{\sqrt{1 - \delta}} \right) \right]^{\delta} \leq \dfrac{1}{4} - \varepsilon,
\]
then the number of vertices of $G$ having at least as many edges to $V \setminus R$ as to $R$ is at most $\delta n$.  In particular, if $\varepsilon = 10^{-10}$ and $\delta = 0.499$, we may take $\alpha = 0.85$.
\end{proposition}
For $\delta = 0.499$, a result such as the above is not difficult to prove for sufficiently large $\alpha$, and in fact versions of this have also appeared in \cite{benjamini2016,tran-vu}.  But we state the above so as to improve the $\Delta = 10$ case of $G_{n,1/2}$ as studied in Tran and Vu.  In that setting, the constants involved are frustratingly important, and a proof with $\alpha = 0.85$ becomes necessary.

\subsubsection*{Applications}
From Theorem \ref{good CLT}---and in fact merely from our variance estimate---we readily obtain:

\begin{corollary}\label{theorem:the main result}
There is a universal constant $C > 0$ such that the following holds.  Suppose we fix any initial red-blue partition $R_0$, $B_0$ of $n$ vertices where $|R_0| - |B_0| \geq \Delta > 0$, and we independently sample a random graph $G_{n,p}$.  Then for all real $t \geq 1$, we have
\[
\mathbb{P}\left(|R_1| \geq n \Phi\left( \dfrac{\Delta p}{\sqrt{np(1-p)}} \right) - \dfrac{nt}{\sqrt{np(1-p)}} \right) \geq 1 - \dfrac{Cp(1-p)}{t^2}.
\]
Moreover, if $5 \leq \Delta p$ and $\sqrt{np(1-p)}\geq 25$, then we have
\[
\mathbb{P}\Big(|R_1| \geq n/2 + \min\left(0.3n , 0.11 \Delta \sqrt{np/(1-p)} \right) \Big) \geq 1 - \dfrac{C}{\Delta}.
\]
\end{corollary}

The above gives us sufficient control of $|R_1|$ to immediately improve several results of Benjamini, Chan, O'Donnell, Tamuz, and Tan \cite{benjamini2016}, which we summarize below.

\begin{theorem}\label{theorem:improving Benjamini, ODonnell}
For each $\varepsilon > 0$, there is a $C > 0$ such that the following holds.  Let $G \sim G_{n,p}$, and initially assign each vertex one of two colors (independently uniformly at random).  Without loss of generality, suppose this assignment initially colors at least half of the vertices red.
\begin{itemize}
\item[(a)] If $p \geq C / \sqrt{n}$, then with probability at least $1- \varepsilon$, the majority dynamics process on $G$ will be unanimously red after at most four steps.
\item[(b)] If $p \geq C / n^{1/3}$, then with probability at least $1-\varepsilon$, the majority dynamics process on $G$ will be unanimously red after at most three steps.
\item[(c)] If $p \geq 1 - C^{-1}$, then with probability at least $1- \varepsilon$, the majority dynamics process on $G$ will be unanimously red after at most two steps.
\end{itemize}
\end{theorem}
A proof of part (a) was the main theorem of \cite{whpFolks}, but results (b) and (c) are new.  Parts (a) and (b) directly generalize Theorems 2 and 3.9 (respectively) of \cite{benjamini2016}, where these events were only proven to hold with probability at least 0.4.  That said, essentially the only difference in our proof is that we are able to use Corollary \ref{theorem:the main result}, which strengthens a weaker result of their paper (namely Theorem 3.5 of \cite{benjamini2016}).  Part (c) of our result also follows in this same way, and we include it here mostly for completeness.

As a second application, the following improves upon a recent paper of Tran and Vu \cite{tran-vu}.

\begin{theorem}\label{theorem:lead of 3 will win}
Suppose we have a red-blue coloring of a vertex set $V = R_0 \cup B_0$ such that there are initially at least $\Delta >0$ more red vertices than blue.  We then independently sample a random graph $G_{|V|,1/2}$.  For $\lambda \in \{\red, \blue\}$, let $P(\lambda)$ denote the probability that the majority dynamics process on $G$ eventually results in the vertices all being colored $\lambda$.  Then
\[
P(\red) - P(\blue) \geq -o(1) + \dfrac{2}{\sqrt{2\pi}} \int_{0} ^{2(\Delta/\sqrt{2\pi} - 0.85)} e^{-x^2 /2} dx = 2 \Phi \left(2(\Delta/\sqrt{2\pi} - 0.85) \right) - 1 - o(1).
\]
In particular, if $\Delta \geq 3$, and $|V|$ is sufficiently large, then we have $P(\red) - P(\blue) \geq 0.511$.
\end{theorem}
Interpreting this in a somewhat playful light, the above result argues that if majority dynamics represents the propagation of opinions in a social network, then one side can obtain a noticeable advantage merely by bribing three extra vertices.  This improves on a result of Tran and Vu \cite{tran-vu} who argued a similar claim but with an initial lead of at least $10$.

The bottleneck in our argument is our reliance on Proposition \ref{lemma:sex panther}, for which the value of $\alpha$ surprisingly becomes important.  We make no claim that our constant $\alpha$ is especially close to the best possible, and sufficiently improving this constant would immediately prove a version of Theorem \ref{theorem:lead of 3 will win} for an initial lead of $1$ vertex.  However, it is easy to imagine that such incremental refinements might be somewhat unsatisfactory, and we instead propose the following, which gets to the heart of the matter.

\begin{conjecture}\label{conjecture:one always suffices}
For each fixed $p \in (0,1)$, there is a corresponding value of $\varepsilon > 0$ such that if there are initially more red vertices than blue, then with probability at least $0.5 + \varepsilon - o(1)$, the majority dynamics process on $G_{n,p}$ will result in all vertices being red.
\end{conjecture}

This conjecture, if true, would likely require a new approach, and we also believe the following stronger conjecture is true (which, in light of our central limit theorem, would exactly determine the value of $\varepsilon$ in Conjecture \ref{conjecture:one always suffices}).

\begin{conjecture}
Given any initial coloring of vertices and any probability $p \geq (1+\varepsilon)\log(n)/n$, if we sample $G_{n,p}$ (independently of the initial coloring) then with probability tending to 1, whichever color has more vertices after the first step of the majority dynamics process will have increasingly more vertices after each subsequent step until this is the unanimous color.
\end{conjecture}
Simulations suggest that both conjectures are true, but we are unable to prove either even when $p$ is constant.  One heuristic argument is that after the first step, one side is likely to have a relatively extreme advantage (e.g., in the case of $p$ constant, after one step one side will be winning by $\Omega(\sqrt{n})$).  If we were to resample the random graph at this point, then such an advantage would be nearly impossible to squander.  We aren't sure how to make this argument rigorous, but we think a resolution of this either way would be very interesting.

\subsection*{Outline of paper}
We begin in Section \ref{sect:lemmas} with a brief collection of known technical results.  In Section \ref{sect:good CLT proof}, we present our Fourier-analytic proof of our main result, Theorem \ref{good CLT}.  We continue in Section \ref{sect:prop and cor} with Proposition \ref{lemma:sex panther} and Theorem \ref{theorem:the main result}.  We then show quick applications in Section \ref{sect:applications} by deriving Theorems \ref{theorem:improving Benjamini, ODonnell} and \ref{theorem:lead of 3 will win}.  Finally, for the sake of clarity, we prove several lemmas in an Appendix.

\section{Known technical lemmas}\label{sect:lemmas}
We will need the following.  Throughout, we let $Bin(m,p)$ denote a binomial random variable with mean $mp$ and variance $mp(1-p)$.

\medskip

\noindent \textbf{Bernstein's inequality:} Let $X_1, X_2, \ldots, X_n$ be independent random variables with $a_i \leq X_i \leq a_i + M$ for all $i$.  Then for all $t \geq 0$, we have
\[
\mathbb{P}\left( \left| \sum_{i=1} ^{n} \left(X_i - \mathbb{E}X_i \right) \right| > t \right) \leq 2 \exp \left( \dfrac{-2 t^2/2}{\sum_{i=1} ^{n} \text{Var}(X_i) + Mt/3} \right).
\]

\begin{lemma}\label{lemma:stupid result} Let $W = Bin(n, p) - Bin(m,p)$ with $p \notin \{0,1\}$.  There is a universal constant $C$ (independent of $m$, $n$, and $p$) such that the following holds.
\begin{itemize}
\item[(i)] For all $t$, we have $\mathbb{P}(W=t) \leq \dfrac{C}{\sqrt{(m+n)p(1-p)}}$.
\item[(ii)] For all $t$, we have $|\mathbb{P}(W=t+1) - \mathbb{P}(W=t)| \leq \dfrac{C}{(m+n)p(1-p)}$.
\end{itemize}
\end{lemma}

\begin{proof}
For each of these claims, we may assume that $m = 0$.  To see this, we would first condition on the value of either $Bin(n,p)$ or $Bin(m,p)$ (whichever has smaller variance), and then the claim reduces to that of a single binomial (with a different value of $t$ and $C$).

When $W = Bin(n,p)$, the first claim is a particularly well-known anticoncentration result.  As for the second claim, consider $F(t) = \mathbb{P}(W=t+1) - \mathbb{P}(W=t)$.  Then by routine manipulation, we see $F(t) - F(t-1)$ is positive iff $|t + 1/2 - p(n+1)| < \sqrt{1 + 4p(1-p)(n+1)}/2$.  Thus, we need only verify the claim for values of $t$ with $|t-np| \leq 2 \sqrt{np(1-p)}$.  For these, we have
\[
F(t) = \mathbb{P}(W=t) \left[ -1 + \dfrac{\mathbb{P}(W=t+1)}{\mathbb{P}(W=t)} \right] = \mathbb{P}(W=t) \left[ -1 + \dfrac{(n-t)p}{(t+1)(1-p)} \right].
\]
We bound $|\mathbb{P}(W=t)| \leq C/\sqrt{np(1-p)}$ by (i), and for $t$ in the given range, we readily see that the second term is also at most $C/\sqrt{np(1-p)}$.
\end{proof}

\section{Proof of Theorem \ref{good CLT}}\label{sect:good CLT proof}
For each $v \in V$, define
\[
\mu_{v} = \begin{cases}2\mathbb{P}\Big(Bin(|R_0|-1,p) \geq Bin(|B_0|,p) \Big)-1 \qquad &\text{if $v \in R_0$},\\2\mathbb{P}\Big(Bin(|B_0|-1,p) \geq Bin(|R_0|,p) \Big)-1 \qquad &\text{if $v \in B_0$},
\end{cases}
\]
and
\[
\varepsilon(v) = \begin{cases}1, \qquad &\text{if $v \in R_0$,}\\-1, \qquad &\text{if $v \in B_0$}. \end{cases}
\]
We will use the notation that $R_1$ (resp.\ $B_1$) denotes the set of red (resp.\ blue) vertices after one step of majority dynamics.  Then let $Z_v$ denote the following (centered) indicator
\[
Z_v = \begin{cases}1-\mu_v \qquad &\text{if $v \in R_0 \cap R_1$ or $v \in B_0 \cap B_1$},\\-1-\mu_v \qquad &\text{otherwise},
\end{cases}
\]
so that by our choice of $\mu_v$, we have $\mathbb{E}[Z_v] = 0$.  Finally, define $Z$ as
\[
Z = \sum_{r \in R_0} Z_r - \sum_{b \in B_0} Z_b = \sum_{v \in V} \varepsilon(v) Z_v = 2|R_1| - n - \mu_r |R_0| + \mu_b |B_0|.
\]
Since $Z$ is an affine transformation of $|R_1|$, we need only prove a central limit theorem for $Z$.  In fact, for each fixed $k$, we will prove
\begin{equation}\label{moments of Z}
\mathbb{E}\left[Z^k \right] = \begin{cases}\mathcal{O}\Big(n^{k/2}(1/\sigma + \Delta/n)\Big) \qquad &\text{if $k$ is odd}\\
\mathcal{O}\Big(n^{k/2}(1/\sigma + \Delta/n)\Big) + (k-1)!! (4n \mu)^{k/2} \qquad &\text{if $k$ is even,}
\end{cases}
\end{equation}
where the implied constants depend only on $k$.

\subsubsection*{Derivation of the theorem statement assuming \eqref{moments of Z}}
Since $\mathbb{E}[Z] = 0$, establishing \eqref{moments of Z} would immediately prove statement (i) of the theorem.  As for statement (ii), suppose $\log(1/\mu) = o(\log(\sigma))$ and $\sigma \to \infty$.  By Bernstein's inequality, we have
\[
\mu \leq \mathbb{P}\Big(|Bin(|R_0|,p) - Bin(|B_0|,p) - \Delta p| \geq \Delta p \Big) \leq 2\exp \left[ \dfrac{- (\Delta p)^2 / 2}{np(1-p) + \Delta p /3} \right].
\]
Since $\sigma \to \infty$, and $\log(1/\mu) = o(\log(\sigma))$, this would imply $\Delta p \ll \sigma \sqrt{\log(\sigma)}$.  Thus, we would need to have $\log(\sigma) = \mathcal{O}( \log(n/\Delta))$.  Therefore, $\eqref{moments of Z}$ would imply that for all fixed $k$
\[
\dfrac{\mathbb{E} \left[ Z ^k \right]}{(4n\mu) ^{k/2}} = \begin{cases}\mathcal{O}\Big((1/\mu^{k/2})(1/\sigma + \Delta/n)\Big) \qquad &\text{if $k$ is odd}\\
\mathcal{O}\Big((1/\mu^{k/2})(1/\sigma + \Delta/n)\Big) + (k-1)!! \qquad &\text{if $k$ is even.}
\end{cases}
\]
And since $\log(1/\mu) = o(\log(\sigma))$ and $\log(1/\mu) = o(\log(n/\Delta))$, these error terms converge to zero for any fixed $k$.  Thus, the moments of $Z/\sqrt{4n\mu}$ converge to those of the standard Gaussian, which establishes a central limit theorem by the method of moments.

Finally, we need to argue that if $\Delta p = \mathcal{O}(\sigma)$ and $\sigma \to \infty$, then the hypotheses of (ii) hold.  For this, we simply note if $\Delta p = \mathcal{O}(\sigma)$, then $\mu$ is bounded below,\footnote{This is by an immediate application of the usual central limit theorem for the binomial random variables in the definition of $\mu$.} and $\log(1/\mu) = \mathcal{O}(1)$.

\subsubsection*{Set up for establishing \eqref{moments of Z}}
Our proof of \eqref{moments of Z} is via standard techniques of ($p$-biased) Fourier analysis.  We first give a brief review and relevant set up, but for more, we refer the reader to O'Donnell's helpful text \cite{booleanFunctions}.

Let $\mathcal{E} = {{V \choose 2}}$ and consider the space of functions mapping from $\{-1,1\}^\mathcal{E} \to \mathbb{R}$.  Each input string is indexed by the $2$-element subsets of $V$, and each $\vec{x} \in \{-1,1\}^\mathcal{E}$ corresponds to a graph whose edge set is $\{e \ : \ x_e = 1\}$.  Consider the product measure on $\{-1,1\}^\mathcal{E}$ where each bit is independently equal to $1$ with probability $p$.  With this, we identify random variables depending on $G \sim G_{|V|,p}$ and functions $f: \{-1,1\}^\mathcal{E} \to \mathbb{R}$ in the obvious way.

For each $S \subseteq \mathcal{E}$, we define the function $\Phi_{S} : \{-1,1\}^\mathcal{E} \to \mathbb{R}$ via
\[
\Phi_{S} (\vec{x}) = \prod_{e \in S} \dfrac{x_e +1 - 2p}{2 \sqrt{p(1-p)}},
\]
with $\Phi_{\emptyset} = 1$.  Then it can be checked that for all $S, T$ we have
\[
\mathbb{E}[\Phi_{S}(\vec{x}) \Phi_{T}(\vec{x})] = \begin{cases}1 \qquad &\text{if $S=T$,}\\
0 \qquad &\text{otherwise.}\end{cases}
\]
These $\Phi_S$ therefore form an orthogonal basis for this space (with inner product $\langle f, g \rangle = \mathbb{E}[fg]$), and each $f$ can uniquely be written as
\[
f(\vec{x}) = \sum_{S \subseteq \mathcal{E}} \widehat{f}(S) \Phi_{S} (\vec{x}), \qquad \text{ where } \qquad \widehat{f}(S) = \mathbb{E}[f(\vec{x}) \Phi_{S}(\vec{x})].
\]
We will need the following Lemma, which we prove in an Appendix.
\begin{lemma}\label{lemma:six Fourier facts}
With notation as before, 
\begin{itemize}
\item[(i)] We have $\mathbb{E}[Z_{v} ^2] = 1- \mu_v ^2 = 4\mu + \mathcal{O}(1/\sigma)$.
\item[(ii)] Let $\Gamma_{u} = \{ \{u, v\} : v \in V \setminus \{u\}\} \subseteq \mathcal{E}$.  If $S \not \subseteq \Gamma_v$, then $\widehat{Z_{v}}(S) = 0$.  Moreover, $\widehat{Z_v}(\emptyset) = 0$.
\item[(iii)] $\widehat{Z_v} (S) = \mathcal{O}(1/\sqrt{n})$ if $|S| \leq 10 k^2$
\item[(iv)] $\widehat{Z_v} (S) = \mathcal{O}(1/n)$ if $2 \leq |S| \leq 10 k^2$
\item[(v)] If $r,b,v$ are distinct with $r \in R_0$ and $b \in B_0$, then $\displaystyle \sum_{e \in \mathcal{E}} \Big(\widehat{Z_r} (e)-\widehat{Z_b}(e) \Big)\widehat{Z_v} (e) = \mathcal{O}\left( \frac{1}{n \sigma} \right)$
\item[(vi)] If $S \neq \emptyset$ and $L \geq 1$, then $\Big | \widehat{\ Z_v ^{L}\ } (S)\Big| \leq 2^L \Big | \widehat{Z_v} (S)\Big| \leq 2^{L+1}$
\end{itemize}
\end{lemma}

Let's prove the moment estimate assuming Lemma \ref{lemma:six Fourier facts}.

\subsection*{Estimating $\mathbb{E}[Z^k]$ assuming Lemma \ref{lemma:six Fourier facts}}
For each tuple $\vec{w} = (w_1, \ldots, w_k) \in V^k$ consisting of (not necessarily distinct) vertices, set $Z_{\vec{w}} = \prod_{i=1} ^{k} Z_{w_i}$ and $\varepsilon(\vec{w}) = \prod_{i=1} ^{k} \varepsilon(w_i)$.  In this notation, we have $\mathbb{E}[Z^k] = \sum_{\vec{w} \in V^k} \varepsilon(\vec{w}) \mathbb{E}[Z_{\vec{w}}]$.

\begin{itemize}
\item For each $c \leq k$, let $\Lambda_{c} \subseteq V^k$ denote the strings in which there are exactly $c$ vertices each appearing exactly once.
\end{itemize}
We will argue that for $\mathbb{E}[Z^k]$, the contribution due to $\Lambda_0$ is the main term, and all of the others are small.  Namely, we'll prove
\[
\sum_{\vec{w} \in \Lambda_0} \varepsilon(\vec{w}) \mathbb{E}[Z_{\vec{w}}] = \begin{cases}\mathcal{O}\left(n^{(k-1)/2} \right) \quad &\text{if $k$ is odd,}\\
\mathcal{O}\left(n^{(k-1)/2} \right) + (k-1)!! \cdot \Big[|R_0| (1-\mu_r ^2) + |B_0| (1-\mu_b ^2) \Big]^{k/2} \quad &\text{if $k$ is even}.\\
\end{cases}
\]
And moreover for all $c \geq 1$, we'll show $\sum_{\vec{w} \in \Lambda_c} \varepsilon(\vec{w}) \mathbb{E}[Z_{\vec{w}}] = \mathcal{O}(n^{k/2}(1/\sigma + \Delta/n + 1/\sqrt{n}))$.

Together, this will yield \eqref{moments of Z} since Lemma \ref{lemma:six Fourier facts}.(i) implies $\Big[|R_0| (1-\mu_r ^2) + |B_0| (1-\mu_b ^2) \Big]^{k/2}$ is of the form $n^{k/2} (4\mu \pm \mathcal{O}(1/\sigma))^{k/2} = (4n\mu)^{k/2} \pm \mathcal{O}(n^{k/2}/\sigma)$ [valid since $4\mu + \mathcal{O}(1/\sigma)$ is bounded].

\subsubsection*{Expansion of $\mathbb{E}[Z_{\vec{w}}]$}
Suppose $T$ and $U$ are disjoint sets of vertices such that $|T| + |U| \leq k$, and suppose for each $u \in U$, there is an integer $\alpha_u \in [2,k]$.  Then we have
\[
\mathbb{E}\left[ \prod_{t \in T} Z_{t} \cdot \prod_{u \in U} Z_{u} ^{\alpha_u} \right] = \sum_{(S_v)_{v \in T \cup U}} \prod_{t \in T} \widehat{\ Z_{t} \ }(S_t) \prod_{u \in U} \widehat{\ Z_{u} ^{\alpha_u} \ }(S_u) \cdot \mathbb{E}\left[ \prod_{v \in T \cup U} \Phi_{S_v} \right],
\]
where the sum is taken over all choices of $(S_v)_{v \in T \cup U}$ in $\mathcal{E}^{|T| + |U|}$ (so that each $v \in T \cup U$ is independently given a subset of $\mathcal{E}$).

Let $\mathcal{H} \subseteq E$ denote the $2$-element subsets of $T \cup U$, and suppose $(S_v)_{v \in T \cup U}$ corresponds to a term in the above summation that isn't zero.  By Lemma \ref{lemma:six Fourier facts}.(ii), for each $v$ we must have $S_v \subseteq \Gamma_v$ (since otherwise one of the Fourier coefficients would be 0).  Moreover every edge of $\bigcup_{u} S_u$ must appear in at least two sets $S_v$ (otherwise the expectation of the $\Phi$ functions would be 0).  Together, this implies that each edge of $\bigcup_{u} S_u$ appears in exactly two sets, that $\bigcup_{u} S_u \subseteq \mathcal{H}$ and that for all $v$ we have $S_v = \left[ \bigcup_{u} S_u \right] \cap \Gamma_v$.  Thus, the above summation becomes
\[
\mathbb{E}\left[ \prod_{t \in T} Z_{t} \cdot \prod_{u \in U} Z_{u} ^{\alpha_u} \right] = \sum_{H \subseteq \mathcal{H}} \prod_{t \in T} \widehat{\ Z_{t} \ }(H \cap \Gamma_t) \prod_{u \in U} \widehat{\ Z_{u} ^{\alpha_u} \ }(H \cap \Gamma_u).
\]
Since $|\mathcal{H}| \leq 10 k^2$, Lemma \ref{lemma:six Fourier facts}.(iii) implies that each term $\widehat{Z_t}(H \cap \Gamma_t)$ is $\mathcal{O}(1/\sqrt{n})$.  Therefore, for any fixed $H \subseteq \mathcal{H}$ we have
\begin{equation}\label{crude term-wise Fourier bound}
\prod_{t \in T} \widehat{\ Z_{t} \ }(H \cap \Gamma_t) \prod_{u \in U} \widehat{\ Z_{u} ^{\alpha_u} \ }(H \cap \Gamma_u) = \mathcal{O}\left(\dfrac{1}{n^{|T|/2}} \right) \prod_{u \in U} \widehat{\ Z_{u} ^{\alpha_u} \ }(H \cap \Gamma_u) = \mathcal{O}\left(n^{-|T|/2} \right).
\end{equation}
If there is a value $u \in U$ for which $H \cap \Gamma_u \neq \emptyset$, then we could apply part (vi) [and (iii)] of Lemma \ref{lemma:six Fourier facts} to show the above expression is $\mathcal{O}(n^{-(|T|+1)/2})$.  Similarly, if there is a vertex $t \in T$ for which $|H \cap \Gamma_t| \neq 1$, then we could apply (iv) to again conclude this expression is $\mathcal{O}(n^{-(|T|+1)/2})$.  Together, since there are at most $2^{|\mathcal{H}|} = \mathcal{O}(1)$ choices for $H$, this implies
\begin{equation}\label{Fourier expansion equation}
\mathbb{E}\left[ \prod_{t \in T} Z_{t} \cdot \prod_{u \in U} Z_{u} ^{\alpha_u} \right] = \mathcal{O}(n^{-(|T|+1)/2}) + \left[ \prod_{u \in U} \widehat{\ Z_{u} ^{\alpha_u} \ }(\emptyset) \right]\cdot  \sum_{H \in \mathcal{M}} \prod_{t \in T} \widehat{Z_t}(H \cap \Gamma_t),
\end{equation}
where $\mathcal{M}$ is the set of all perfect matchings on the vertex set $T$.

\subsubsection*{Contribution from $\Lambda_0$}
Let $\Lambda_0 ^{+} \subseteq \Lambda_0$ denote the set of strings where each vertex appears at most twice.  Then we have
\[
\sum_{\vec{w} \in \Lambda_0} \varepsilon(\vec{w}) \mathbb{E}[Z_{\vec{w}}] = \sum_{\vec{w} \in \Lambda_0 ^+} \mathbb{E}[Z_{\vec{w}}] + \mathcal{O}(|\Lambda_0 \setminus \Lambda_0 ^+|) = \sum_{\vec{w} \in \Lambda_0 ^+} \mathbb{E}[Z_{\vec{w}}] + \mathcal{O}(n^{(k-1)/2}),
\]
which holds since each term of the summation is at most $\mathcal{O}(1)$ and each string in $\Lambda_0 \setminus \Lambda_0 ^+ \subseteq V^k$ has an element appearing more than twice (and none appearing only once).  When $k$ is odd, $\Lambda_0 ^+ = \emptyset$, so the above summation is $\mathcal{O}(n^{(k-1)/2})$.  Otherwise, for $k$ even, \eqref{Fourier expansion equation} [together with Lemma \ref{lemma:six Fourier facts}.(i)] implies
\begin{eqnarray*}
\sum_{\vec{w} \in \Lambda_0 ^+} \mathbb{E}[Z_{\vec{w}}] &=& \sum_{U \subseteq V,\ \ |U|=k/2} \dfrac{k!}{2^{k/2}}\mathbb{E}\left[ \prod_{u \in U} Z_{u} ^{2} \right] = \mathcal{O}\left(n^{(k-1)/2}\right) + \sum_{U \subseteq V,\ \ |U|=k/2} \dfrac{k!}{2^{k/2}} \prod_{u \in U} \widehat{\ Z_{u} ^{2}\ } (\emptyset)\\
&=& \mathcal{O}\left(n^{(k-1)/2}\right) + \dfrac{k!}{2^{k/2}} \sum_{j=0}^{k/2} {|R_0| \choose j} {|B_0| \choose {k/2 -j}} (1-\mu_r ^2)^{j} (1-\mu_b ^2)^{k/2 -j}\\
&=& \mathcal{O}\left(n^{(k-1)/2}\right) + \dfrac{k!}{2^{k/2}} \left( \dfrac{1}{(k/2)!} \Big[ |R_0| (1-\mu_r ^2) + |B_0| (1- \mu_b ^2) \Big]^{k/2} + \mathcal{O}\left(n^{k/2 -1}\right) \right)\\
&=& \mathcal{O}\left(n^{(k-1)/2}\right) + (k-1)!! \cdot \Big[ |R_0| (1-\mu_r ^2) + |B_0| (1- \mu_b ^2) \Big]^{k/2}.
\end{eqnarray*}

\subsubsection*{Contribution from $\Lambda_c$ for $c \geq 1$}
For ease of notation, let us now assume without loss of generality that $|R_0| \geq |B_0|$, and let us write $R_0 = \{r_1, r_2, \ldots, r_m\} \cup X$ and $B_0 = \{b_1, b_2, \ldots, b_m\}$ where $|X| = \Delta = |R_0| - |B_0|$.  For each $c \geq 1$, let $\Lambda_c ^+$ denote the strings $\vec{w} \in \Lambda_c$ such that (a) $\vec{w}$ contains no elements of $X$, (b) $\vec{w}$ contains no elements more than twice, and (c) for each $i$, if $r_i$ appears in $\vec{w}$, then $b_i$ does not.  Then using \eqref{crude term-wise Fourier bound} we have
\begin{eqnarray*}
\sum_{\vec{w} \in \Lambda_c \setminus \Lambda_c ^+} \varepsilon(\vec{w}) \mathbb{E}[Z_{\vec{w}}] &=& |\Lambda_c \setminus \Lambda_c ^+| \cdot \mathcal{O}(n^{-c/2}) = \mathcal{O}\left(n^{k/2 - 1} (\Delta +\sqrt{n}) \right),
\end{eqnarray*}
which follows since
\[
|\Lambda_c \setminus \Lambda_c ^+| = \mathcal{O}\left(n^{(k+c)/2 -1} |X| \right) + \mathcal{O}\left(n^{(k+c-1)/2} \right) + \mathcal{O}\left(n^{(k+c)/2 -1} \right).
\]
Moreover, we have
\[
\sum_{\vec{w} \in \Lambda_c ^+} \varepsilon(\vec{w}) \mathbb{E}[Z_{\vec{w}}] = \dfrac{k!}{2^{(k-c)/2}} \sum_{U \subseteq V} \sum_{T \subseteq V \setminus U} \mathbb{E}\left[ \prod_{t \in T} (Z_{r_t} - Z_{b_t}) \prod_{u \in U} (Z_{r_u} ^2 + Z_{b_u} ^2) \right].
\]

Fixing $U$ and $T$ and expanding this product out as in \eqref{Fourier expansion equation}, we see
\[
\mathbb{E}\left[ \prod_{t \in T} (Z_{r_t} - Z_{b_t}) \prod_{u \in U} (Z_{r_u} ^2 + Z_{b_u} ^2) \right] = \mathcal{O}\left(n^{-(c+1)/2} + \sum_{H \in \mathcal{M}} \prod_{t \in T} [\widehat{Z_{r_t}}(H \cap \Gamma_{r_t}) - \widehat{Z_{b_t}} (H \cap \Gamma_{b_t})] \right),
\]
where $\mathcal{M}$ is the set of graphs $H$ on the vertex set $\bigcup_{t \in T} \{r_t, b_t\}$ such that for all $t \in T$, we have $r_t \not \sim b_t$ and there is exactly one edge of $H$ meeting $\{r_t, b_t\}$.  For distinct $i,j \in T$, let $\mathcal{M}_{i,j} \subseteq \mathcal{M}$ denote the subset of graphs containing one of the edges in $\{r_i, b_i\} \times \{r_j, b_j\}$.  Then we have
\begin{eqnarray*}
\sum_{H \in \mathcal{M}} \prod_{t \in T} [\widehat{Z_{r_t}}(H \cap \Gamma_{r_t}) - \widehat{Z_{b_t}} (H \cap \Gamma_{b_t})]  &=& \dfrac{1}{|T|} \sum_{i,j} \sum_{H \in \mathcal{M}_{i,j}} \prod_{t \in T} [\widehat{Z_{r_t}}(H \cap \Gamma_{r_t}) - \widehat{Z_{b_t}} (H \cap \Gamma_{b_t})]\\
&=& \dfrac{1}{|T|} \sum_{i,j} \sum_{H \in \mathcal{M}_{i,j}} \prod_{t \notin \{i,j\}} [\widehat{Z_{r_t}}(H \cap \Gamma_{r_t}) - \widehat{Z_{b_t}} (H \cap \Gamma_{b_t})]\\
& & \qquad \times \left(\dfrac{1}{4} \sum_{e} \Big[\widehat{Z_{r_i}}(e) - \widehat{Z_{b_i}}(e) \Big] \cdot \Big[\widehat{Z_{r_i}}(e) - \widehat{Z_{b_i}}(e) \Big] \right)\\
&=& \mathcal{O} \left( n^{1-c/2} \right) \cdot \mathcal{O}\left( \dfrac{1}{n \sigma} \right) = \mathcal{O}(n^{-c/2} / \sigma),
\end{eqnarray*}
where the last line follows from Lemma \ref{lemma:six Fourier facts}.(v).  Therefore, we have
\[
\sum_{\vec{w} \in \Lambda_c} \varepsilon(\vec{w}) \mathbb{E}[Z_{\vec{w}}] = \sum_{\vec{w} \in \Lambda_c \setminus \Lambda_c ^{+}} \varepsilon(\vec{w}) \mathbb{E}[Z_{\vec{w}}] + \sum_{\vec{w} \in \Lambda_c ^+} \varepsilon(\vec{w}) \mathbb{E}[Z_{\vec{w}}] = \mathcal{O}\left(n^{k/2 - 1} (\Delta +\sqrt{n}) \right) + \mathcal{O}\left(n^{k/2} / \sigma \right),
\]
as desired. $\qed$

\section{Derivations of Proposition \ref{lemma:sex panther} and Corollary \ref{theorem:the main result}}\label{sect:prop and cor}
\subsection*{Proof of Proposition \ref{lemma:sex panther}}
\begin{proof}
For $G \sim G_{n,p}$, call a pair of sets $(R,U)$ \textit{bad} iff (1) $|R| \geq n/2 + \alpha \sqrt{n(1-p)/p}$, (2) $|U| = \delta n$, and (3) every vertex of $U$ has at least as many neighbors in $V \setminus R$ as it does in $R$.  For each pair $(R,U)$, we will show under the hypotheses of the proposition that if $\sqrt{np(1-p)}$ is large enough (depending only on $\varepsilon$), there is a value $q < 1$ such that $\mathbb{P}((R,U) \text{ is bad}) \leq (q/4)^{n}$, which will prove that with probability at least $1 - q^n$, the desired conclusion holds for every set $R$ and every $\delta \in (\varepsilon, 1-\varepsilon)$.

Let $R, U$ be fixed satisfying conditions (1) and (2) above, and define $B = V \setminus R$.  Note that if we condition on the edges in $G[U]$, then the events ``$d_{R} (u) \leq d_{B} (u)$'' are independent as $u$ ranges over $U$. Thus, we have
\begin{eqnarray*}
\mathbb{P}((R,U) \textit{ is bad}\ |\ G[U]) &=& \prod_{u \in U} \mathbb{P} \Big( d_{R \setminus U} (u) - d_{B \setminus U} (u) \leq d_{B \cap U} (u) - d_{R \cap U} (u) \ | \ G[U] \Big)\\
&\leq & \mathbb{P} \left( d_{R \setminus U} (u) - d_{B \setminus U} (u) \leq 1 + \dfrac{2e(B \cap U) - 2 e(R \cap U)}{|U|} \ | \ G[U] \right)^{|U|},
\end{eqnarray*}
where $e(S)$ denotes the number of edges of $G[S]$, and this inequality holds due to the log-concavity of binomial distributions.\footnote{Namely, if $Q(x) = \mathbb{P}(Bin(m,p) \leq x)$, then a routine computation shows for any $x < y$ we have $Q(x)Q(y) \leq Q(x+1)Q(y-1)$.  Thus, this concavity holds for linear combinations of binomial random variables as well.}  

Letting $M$ denote the random variable $M = 1+ \dfrac{2e(B \cap U) - 2e(B \cap U)}{|U|}$, we have
\begin{equation}\label{equation:set up with expected value for bad quadruples}
\mathbb{P}((R,U) \textit{ bad}) \leq \sum_{t} \mathbb{P} \left(M = t \right) \left[ \mathbb{P} \Big( Bin(|R \setminus U|, p) - Bin(|B \setminus U|, p) \leq t \Big) \right]^{|U|}.
\end{equation}
The mean of $M$ is easily computed as
\begin{eqnarray*}
\mathbb{E}[M - 1] / p &=& \dfrac{2}{|U|}\left[{|B \cap U| \choose 2} - {|R \cap U| \choose 2} \right] = |U| - 2|R \cap U|-1 + 2|R \cap U|/|U|\\ 
&=& |B| - |R| + |R \setminus U| - |B \setminus U| - 1 + 2|R \cap U|/|U|.
\end{eqnarray*}
And for any fixed $0 < \gamma \leq 2$, setting $\Lambda = \sqrt{np(1-p)}$, Bernstein's inequality implies that
\begin{eqnarray*}
\mathbb{P}(M \geq \mathbb{E}[M] + \gamma \Lambda) 
&\leq& \exp \left(\dfrac{-|U|^2 \gamma^2 n p (1-p)/8}{\left[\cho{|R \cap U|}{2} + \cho{|B \cap U|}{2} \right] p(1-p) + |U| \gamma \sqrt{np(1-p)}/6} \right)\\
&\leq& \exp \left(-(1+o(1))\dfrac{|U|^2 \gamma^2 n/4}{|U|^2 - 2|R \cap U| |B \cap U|} \right) \leq \exp \left(-(1+o(1))\dfrac{\gamma^2 n}{2} \right),
\end{eqnarray*}
where the $o(1)$ term tends to $0$ (as $\sqrt{np(1-p)} \to \infty$) uniformly over all $\gamma \in [0, 2]$.

Let $Z \sim Bin(|R \setminus U|, p) - Bin(|B \setminus U|, p)$.  Since $\sqrt{np(1-p)} \to \infty$ and since $|R \setminus U|+|B \setminus U|$ is order $n$, the central limit theorem applies to $Z$, and for all fixed $\gamma \geq 0$ we have
\begin{eqnarray*}
\mathbb{P} \Big( Z \leq \mathbb{E}[M] + \gamma \Lambda \Big) &=& \mathbb{P} \Big( Z \leq \mathbb{E}[Z] + 1 + (|B|-|R|)p - p + 2p|R \cap U|/|U| +\gamma \Lambda \Big)\\
&\leq& o(1) + \mathbb{P} \Big( Z \leq \mathbb{E}[Z] - (2\alpha -\gamma) \Lambda \Big)
= o(1) + \Phi \left(\frac{\gamma - 2\alpha}{\sqrt{1 - |U|/n}} \right).
\end{eqnarray*}
Thus, putting these together we obtain
\begin{eqnarray*}
\mathbb{P}((R,U) \text{ bad}) &\leq& \sum_{\gamma} \mathbb{P} \left(M = \mathbb{E}[M] + \gamma \Lambda \right) \mathbb{P} \Big(Z \leq \mathbb{E}[M] + \gamma \Lambda \Big)^{|U|} \\
&\leq & \mathbb{P} \Big(Z \leq \mathbb{E}[M] \Big)^{|U|} + \mathbb{P} \Big(M \geq \mathbb{E}[M] + 2 \Lambda \Big)\\
& & \qquad + \sum_{0 < \gamma <2} \mathbb{P} \left(M = \mathbb{E}[M] + \gamma \Lambda \right) \mathbb{P} \Big(Z \leq \mathbb{E}[M] + \gamma \Lambda \Big)^{|U|}\\
&\leq& (0.2 +o(1))^{n} + n^{2} \Bigg((1+o(1)) \sup_{0 \leq \gamma \leq 2} e^{-\gamma^2 / 2} \left[\Phi \left(\frac{\gamma - 2\alpha}{\sqrt{1 - |U|/n}} \right) \right]^{|U|/n} \Bigg)^n.
\end{eqnarray*}
Finally, by absorbing smaller terms into the $o(1)$ error, we obtain the desired bound.
\end{proof}

\subsection*{Proof of Corollary \ref{theorem:the main result}}
\begin{proof} We first prove the second part of the corollary given the first.  For this, we assume $5 \leq \Delta p$ and $25 \leq \sqrt{np(1-p)}$.

\textbf{Case 1:} First suppose $\Delta p \leq \sqrt{np(1-p)}$.  Note that for all $x \geq 0$, an application of the mean value theorem implies $\Phi(x) \geq 1/2 + x /(2 \pi e^{x^2})^{1/2}$.  With a degree of foresight, let us set $t = 0.9 (\Delta p) /\sqrt{2 \pi e}$, which is greater than $1$ since $5 \leq \Delta p$.  Since $\Delta p \leq \sqrt{np(1-p)}$, we have
\[
n \Phi \left( \dfrac{\Delta p}{\sqrt{np(1-p)}} \right) - \dfrac{nt}{\sqrt{np(1-p)}} \geq \dfrac{n}{2} + \dfrac{n\Delta p / \sqrt{2 \pi e}}{\sqrt{np(1-p)}} - \dfrac{nt}{\sqrt{np(1-p)}} \geq \dfrac{n}{2} + 0.02\Delta  \sqrt{np/(1-p)}.
\]
Thus, using the first part of the lemma, we obtain
\[
\mathbb{P}\Big(Z \geq n/2 + 0.02 \Delta \sqrt{np/(1-p)} \Big) \geq 1 - \dfrac{C p(1-p)}{t^2} \geq 1 - \dfrac{C}{\Delta}.
\]

\textbf{Case 2:} Now suppose $\Delta p  \geq \sqrt{np(1-p)}$.  Numerical computation shows $\Phi(1) - 0.04 \geq 0.8$, so setting $t = 0.04 \sqrt{np(1-p)}$ (which is in fact at least $1$ by assumption) we have
\[
\mathbb{P}(Z \geq 0.8 n) \geq \mathbb{P}\left(Z \geq n \Phi\left( \dfrac{\Delta p}{\sqrt{np(1-p)}} \right) - \dfrac{nt}{\sqrt{np(1-p)}} \right) \geq 1 - \dfrac{Cp(1-p)}{t^2} \geq 1 - \dfrac{625C}{n}.
\]

Thus, in either case, we have a lower bound of at least $1 - 625C / \Delta$ for each corresponding probability, so in both cases, we know one (if not both) of these probabilities is at least this large, which finishes the proof (with the value of $C$ replaced by $625C$).

\hrulefill

We now turn our attention to proving the first statement.  For each $v \in V,$ let $Z_v$ be the indicator for the event that $v$ will be red after one step of the majority dynamics process.  Then $\mathbb{E}[Z_v] \geq \mathbb{P}(Bin(|R_0|, p) - Bin(|B_0|,p) > 0)$.  Setting $W = Bin(|R_0|,p) - Bin(|B_0|,p)$, we have $\mathbb{E}[W] = \Delta p$, and using a version of the Berry-Esseen theorem from Shevstova \cite{berryEsseen}, we obtain
\begin{eqnarray*}
\mathbb{E}[Z_v] &\geq & \mathbb{P}(W > 0) = \mathbb{P}(W - \mathbb{E}[W] > -\Delta p) = \mathbb{P}\left(\dfrac{W - \mathbb{E}[W]}{\sqrt{n p(1-p)}} > \dfrac{-\Delta p}{\sqrt{n p(1-p)}} \right)\\
&\geq& \Phi \left( \dfrac{\Delta p}{\sqrt{np(1-p)}} \right) - 0.4748 \dfrac{p (1-p)^3 + (1-p)p^3}{\sqrt{(n-1) p^3 (1-p)^3}}\\
&\geq& \Phi \left( \dfrac{\Delta p}{\sqrt{np(1-p)}} \right) - \dfrac{0.475}{\sqrt{np (1-p)}}.
\end{eqnarray*}
Thus, since $|R_1| = \sum_{v} Z_v$, we have
\[
\mathbb{E}|R_1| \geq n \Phi\left( \dfrac{\Delta p}{\sqrt{np(1-p)}} \right) - \dfrac{0.475 \sqrt{n}}{\sqrt{p (1-p)}}.
\]

The variance of $|R_1|$ is computed in Theorem \ref{good CLT} (i) as $\text{Var} |R_1| \leq n \mu +\mathcal{O}(\Delta + n/\sigma) \leq cn$.  Finally, we finish by applying Chebyshev's inequality (with $t \geq 1$) as follows:
\begin{eqnarray*}
\mathbb{P}\left(|R_1| \geq n \Phi\left( \dfrac{\Delta p}{\sqrt{np(1-p)}} \right) - \dfrac{nt}{\sqrt{np(1-p)}} \right) &\geq& \mathbb{P}\left(\mathbb{E}|R_1| - |R_1| \leq \dfrac{(t-0.475)\sqrt{n}}{\sqrt{p(1-p)}} \right)\\
&\geq& \mathbb{P}\left(\mathbb{E}|R_1| - |R_1| \leq \dfrac{t(1-0.475)\sqrt{n}}{\sqrt{p(1-p)}} \right)\\
&\geq& 1 - \dfrac{Cp(1-p)}{t^2}. \qedhere
\end{eqnarray*}
\end{proof}

\section{Applications}\label{sect:applications}
\subsection*{Proof of Theorem \ref{theorem:improving Benjamini, ODonnell}}

\begin{proof}
Each of the three parts of the theorem are proven from Theorem \ref{theorem:the main result} in an analogous way exactly as in \cite{benjamini2016}.  We will treat the three cases simultaneously, and we assume throughout that $n$ is large enough to support our argument (valid by choosing $C$ large enough).

Let $R_0, B_0$ be the initial (uniformly random) coloring of $G$.  Pick a fixed $\gamma \in (0,1)$ small enough so that $\mathbb{P} \left( \Big| |R_0| - |B_0| \Big| \geq \gamma \sqrt{n} \right) > 1 - \varepsilon/2$, which is possible since $|R_0| = n- |B_0|$ is distributed as a binomial random variable with mean $n/2$ and variance $n/4$.  Define the events:
\begin{itemize}
\item $E_0$ is the event that $|R_0| - |B_0| \geq \gamma \sqrt{n}$
\item $E_1$ is the event that $|R_1| \geq n/2 + \min(.03n, 0.11 \gamma n\sqrt{p/(1-p)}) \geq n/2 + 0.11 \gamma n \sqrt{p}$
\item $E_2$ is the event that $|B_2| \leq \dfrac{32}{(0.22 \gamma \sqrt{p})^2 p} \leq \dfrac{700 / \gamma^2 }{p^2}$
\item $E_3$ is the event that $|B_3| \leq \dfrac{32}{(0.8)^2  p} \leq \dfrac{50}{p}$
\item $F$ is the event that every vertex of $G$ has degree at least $np/2$
\end{itemize}

If $p \geq (100/\gamma)/\sqrt{n}$, we claim that all of these events happen simultaneously with probability at least $1/2 - \varepsilon/4 -o(1)$, which---by increasing the constants of the theorem statement to allow us to assume $n$ is sufficiently large---will show they happen with probability at least $1/2 - \varepsilon/2$.  Having shown this, we will be done as follows.
\begin{itemize}
\item[(i)] If $p \geq (100/\gamma)/\sqrt{n}$, then $p \geq 100 / \sqrt{n}$, so we have $200/n < p^2$ implying $50/p < np/4$.  Therefore, assuming $E_3$ and $F$, we would deterministically need $B_4 = \emptyset$.
\item[(ii)] If $p \geq 15 / (\gamma^2 n)^{1/3}$ then $(2800/\gamma^2) / n < p^3$ implying $(700/\gamma^2) / p^2 < np/4$.  Thus, assuming $E_2$ and $F$, we would need to have $B_3 = \emptyset$.
\item[(iii)] If $p \geq 1 - (1+8/\gamma^2)^{-1}$  then $0.11 \gamma \sqrt{p/(1-p)} \geq 0.3$.  Therefore, assuming $E_1$, we would have $|B_1| \leq 0.2n$, which is less than $np/4$ since our assumption on $p$ (and $\gamma < 1$) implies $p > 0.8$.  Therefore, $E_1$ and $F$ would imply $B_2 = \emptyset$.
\end{itemize}

Thus, it only remains to show that if $p \geq (100/\gamma)/\sqrt{n}$, then events $E_0, E_1, E_2, E_3$ and $F$ all simultaneously occur with probability at least $1/2 - \varepsilon/4 - o(1)$.  We prove these in sequential order except for the claim $\mathbb{P}(E_1 | E_0) = 1 - o(1)$, which we present last.

First note that $\mathbb{P}(E_0) \geq 1/2 - \varepsilon/4$ by our choice of $\gamma$.  To show $\mathbb{P}(E_2 | E_1) = 1 - o(1)$ and $\mathbb{P}(E_3 | E_2) = 1 - o(1)$, we refer the reader to Lemmas 3.6 and 3.7 of \cite{benjamini2016}, which prove the following.

\begin{lemma}[Benjamini, Chan, O'Donnell, Tamuz, and Tan \cite{benjamini2016}]\label{benjamini lemma}

If $p \geq \log(n) ^5 / n$, then with probability tending to 1, the graph $G_{n,p}$ will satisfy the following property for all $\alpha > 0$.  If $R \cup B$ is \textbf{any} initial red-blue coloring with $|R| - |B| \geq \alpha n$, then after one step of majority dynamics the number of blue vertices will be at most $32/(\alpha ^2 p)$.
\end{lemma}
Therefore, with high probability we have that $E_1$ will imply $E_2$, and  $E_2$ in turn implies $E_3$ (because $(700/\gamma^2) / p^2 < 0.1 n$).  We control $\mathbb{P}(F)$ by a simple union bound together with Chernoff's inequality to obtain
\[
\mathbb{P}(F) \geq 1 - n \mathbb{P}(Bin(n-1, p) < np/2) \geq 1 - n \exp \left(\dfrac{-np(1/2)^2}{2} \right) = 1 -\exp \left(\log(n) - \dfrac{np}{8} \right),
\]
which tends to $1$ since $p \geq (100/\gamma)/\sqrt{n} \gg \log(n)/n$.

Finally, we must show $\mathbb{P}(E_1 |E_0) \geq 1 - o(1)$.  If $\sqrt{np(1-p)} \geq 25$, then our choice of $p$ immediately allows us to apply Theorem \ref{theorem:the main result}, which implies $\mathbb{P}(E_1 | E_0) \geq 1 - \mathcal{O}((\gamma\sqrt{n})^{-1})$.  On the other hand, if $\sqrt{np(1-p)} < 25$, then we would need to have $p \geq 2/3$ (since $np \to \infty$) and thus $1-p < 2 \cdot (3/2) 25^2/n \leq 1000/n$.  Therefore, with probability at least $1- o(1)$, every vertex in $G$ has degree at least $n - \log(n)$ since
\begin{eqnarray*}
\mathbb{P}(\text{min deg}(G) \leq n - \log(n)) &\leq& n \cdot \mathbb{P}(Bin(n, 1-p) \geq \log(n))\\
&\leq& n \cdot \mathbb{P}(Bin(n, 1000/n) \geq \log(n))\\
&\leq& n \exp\left( -(3/2 + o(1)) \log(n)\right) = o(1),
\end{eqnarray*}
where the last inequality follows from Bernstein's inequality.  Therefore, if $E_0$ occurs, then with probability $1 - o(1)$ we would need $R_1 = V$ (and in particular $\mathbb{P}(E_1 | E_0) = 1 -o(1)$).  This is because with probability $1-o(1)$ every vertex of $G$ has degree at least $n - \log(n)$, so if $E_0$ occurs, then initially every vertex would have at least $|R_0| - |B_0| - \log(n) > 0$ more red neighbors than blue.  (So in this case, majority dynamics will typically end after only one step.)

In any case, regardless of the value of $\sqrt{np(1-p)}$, we've shown $\mathbb{P}(E_1 | E_0) = 1 - o(1)$, completing the proof.
\end{proof}

\subsection*{Proof of Theorem \ref{theorem:lead of 3 will win}}
For this proof, we need the following lemma, as proven in the Appendix.
\begin{lemma}\label{lemma:sway over some vertices}
Let $V = X \cup R \cup B$ where $|R| = |B|$, and suppose we independently sample the random graph $G_{|V|, p}$.  Define the sets $T_R = \{v \in R \ : \ 0 < d_{B} (v) - d_{R} (v) \leq d_{X} (v)\}$ and $T_B = \{v \in B \ : \ 0 \leq d_{B} (v) - d_{R} (v) < d_{X} (v)\}$, and let $\sigma = \sqrt{2|R|p(1-p)}$.  If $1+|X|p = o\Big(\sigma \Big)$ then for any fixed $\varepsilon > 0$, with high probability we have
\[
\Bigg| |T_R \cup T_B| - \mathbb{E}[|T_R \cup T_B|] \Bigg| \leq \sqrt{|V|} \cdot \left(\dfrac{|X| p}{\sigma} \right)^{1/2 - \varepsilon}.
\]
Moreover, $\mathbb{E}[|T_R \cup T_B|] = \dfrac{|V| \cdot |X| p}{\sigma \sqrt{2\pi}} + \mathcal{O}\left( \dfrac{|V| (1+ |X|p)^2}{\sigma ^2}\right)$.
\end{lemma}
Assuming this lemma, we now proceed to a proof of the theorem at hand.
\begin{proof}[Proof of Theorem \ref{theorem:lead of 3 will win}]
Suppose there are initially $\Delta$ more red vertices than blue, and by monotonicity of the majority dynamics process, we may first assume $\Delta = o(n^{1/5})$ [since reducing $\Delta$ only decreases the quantity in question, and this would only change the integral by at most $o(1)$].  Before we run the majority dynamics process, let us consider the following coupling.  We first decompose the vertices of $G$ as $V = X \cup R_0 \cup B_0$, where $|R_0| = |B_0|$ and $|X| = \Delta$.  The vertices in $R_0$ will be initially red, the vertices of $B_0$ initially blue, and the vertices of $X$ will either be initially all red or initially all blue with these two options being equally likely.

Suppose we sample the random graph $G_{n,p}$, but we have not yet decided on the color for $X$.  For each $\lambda \in \{\red, \blue\}$, consider the majority dynamics process that would evolve if $X$ is initially colored $\lambda$, and define $W(\lambda)$ to be either (i) $\tie$ if this process does not lead to unanimity or (ii) the unanimous color that the process eventually attains.

In this language, we seek to bound $\mathbb{P}(W(\red) = \red)- \mathbb{P}(W(\red) = \blue)$.  By monotonicity, we have $\mathbb{P}(W(\red) = \blue \text{ and } W(\blue) \neq \blue) = 0$.  Using this and exploiting symmetry, we have
\begin{eqnarray*}
\mathbb{P}(W(\red) = \red) - \mathbb{P}(W(\red) = \blue) &=&  \mathbb{P}(W(\blue) = \blue \text{ and } W(\red) = \red) + \mathbb{P}(W(\blue) = \blue \text{ and } W(\red) = \tie)\\
&\geq& \mathbb{P}(W(\blue) = \blue \text{ and } W(\red) = \red).
\end{eqnarray*}
Let $|R_i (\lambda)|$ denote the number of red vertices that the graph would have after $i$ steps if $X$ is initially colored $\lambda$.  By Proposition \ref{lemma:sex panther}, with probability $1 - o(1)$, we have the two implications
\[
|R_1 (\red)| \geq n/2 + 0.85 \sqrt{n} \quad \Rightarrow \quad |R_2 (\red)| \geq 0.501 n\quad \Rightarrow \quad |R_3(\red)| \geq 0.999n.
\]
By the same reasoning as in the proof of Theorem \ref{theorem:improving Benjamini, ODonnell}, this in turn would imply $R_4(\red) = V$ with probability tending to 1.  Therefore, with high probability $|R_1 (\red)| \geq n/2 + 0.85 \sqrt{n}$ would imply $W(\red) = \red$ and similarly $|R_1 (\blue)| < n/2 - 0.85 \sqrt{n}$ would imply $W(\blue) = \blue$.

Let $Z_0$ denote the number of vertices that would be red after one step if the set $X$ were completely removed from $G$.  Since $|X| = o(n^{1/5})$, we can use Lemma \ref{lemma:sway over some vertices} (see above) with $\varepsilon = 1/6$ to argue that with high probability
\begin{eqnarray*}
\Bigg| |R_1 (\red)| - |R_1(\blue)| - 2 |X| \sqrt{n/(2\pi)} \Bigg| &\leq& \Big| |R_1 (\red)| - Z_0 - |X| \sqrt{n/(2\pi)} \Big| \\
& & \qquad \qquad + \Big| |R_1 (\blue)| - Z_0 + |X| \sqrt{n/(2\pi)} \Big|\\
&\leq& 2\left(c n^{0.4} +  c|X|^2 + |X|\right) \leq C n^{0.4},
\end{eqnarray*}
for some constants $c$ and $C$ (here, the term $c |X|^2$ is from the error term in the estimate of $\mathbb{E}[|T_R \cup T_B|]$, and $|X|$ is to account for the number of red vertices that would be in $X$ after one step).  Thus, we have
\begin{eqnarray*}
& & \mathbb{P}(W(\red) = \red) - \mathbb{P}(W(\red) = \blue) \geq \mathbb{P}(W(\blue) = \blue \text{ and } W(\red) = \red)\\
& & \quad \qquad \qquad \geq -o(1) + \mathbb{P}\left(0.85 \sqrt{n} \leq  |R_1 (\red)| -n/2 \leq (2|X|/\sqrt{2\pi}-0.85) \sqrt{n} - 2C n^{0.4} \right).
\end{eqnarray*}
Finally, using the central limit law for $|R_1 (\red)|$ given by Theorem \ref{good CLT} together with the fact that $\Delta = |X|$, we obtain that the above expression is at least
\begin{eqnarray*}
\mathbb{P}(W(\red) = \red) - \mathbb{P}(W(\red) = \blue) &\geq& -o(1) + \mathbb{P}\left(\left| \dfrac{|R_1 (\red)| - \mathbb{E}|R_1 (\red)|}{\sqrt{n} /2} \right | \leq 2(|X|/\sqrt{2\pi}-0.85) \right)\\
& \geq & -o(1) + \dfrac{2}{\sqrt{2\pi}} \int_{0} ^{2(\Delta/\sqrt{2 \pi} - 0.85)} e^{-x^2 /2} dx.\qedhere
\end{eqnarray*}
\end{proof}

\bibliographystyle{plain}
\bibliography{myBib.bib}

\begin{thebibliography}{10}

\bibitem{abdullah2015}
Mohammed~Amin Abdullah and Moez Draief.
\newblock Global majority consensus by local majority polling on graphs of a
  given degree sequence.
\newblock {\em Discrete Applied Mathematics}, 180:1--10, 2015.

\bibitem{amir2019majority}
Gideon Amir, Rangel Baldasso, and Nissan Beilin.
\newblock Majority dynamics and the median process: connections, convergence
  and some new conjectures.
\newblock {\em arXiv preprint arXiv:1911.08613}, 2019.

\bibitem{econ2}
Venkatesh Bala and Sanjeev Goyal.
\newblock Learning from neighbours.
\newblock {\em The review of economic studies}, 65(3):595--621, 1998.

\bibitem{balogh2007bootstrap}
J{\'o}zsef Balogh and Boris~G Pittel.
\newblock Bootstrap percolation on the random regular graph.
\newblock {\em Random Structures \& Algorithms}, 30(1-2):257--286, 2007.

\bibitem{benjamini2016}
Itai Benjamini, Siu-On Chan, Ryan O’Donnell, Omer Tamuz, and Li-Yang Tan.
\newblock Convergence, unanimity and disagreement in majority dynamics on
  unimodular graphs and random graphs.
\newblock {\em Stochastic Processes and their Applications}, 126(9):2719--2733,
  2016.

\bibitem{physicis1}
Emilio De~Santis and Charles~M Newman.
\newblock Convergence in energy-lowering (disordered) stochastic spin systems.
\newblock {\em Journal of statistical physics}, 110(1-2):431--442, 2003.

\bibitem{physics2}
Sinziana~M Eckner and Charles~M Newman.
\newblock Fixation to consensus on tree-related graphs.
\newblock {\em ALEA}, 12(1):357--374, 2015.

\bibitem{econ1}
Glenn Ellison and Drew Fudenberg.
\newblock Rules of thumb for social learning.
\newblock {\em Journal of political Economy}, 101(4):612--643, 1993.

\bibitem{whpFolks}
Nikolaos Fountoulakis, Mihyun Kang, and Tam{\'a}s Makai.
\newblock Resolution of a conjecture on majority dynamics: rapid stabilisation
  in dense random graphs.
\newblock {\em arXiv preprint arXiv:1910.05820}, 2019.

\bibitem{colorWar}
Bernd G{\"a}rtner and Ahad N.~Zehmakan.
\newblock Color war: Cellular automata with majority-rule.
\newblock In Frank Drewes, Carlos Mart{\'i}n-Vide, and Bianca Truthe, editors,
  {\em Language and Automata Theory and Applications}, pages 393--404, Cham,
  2017. Springer International Publishing.

\bibitem{gartner2018majority}
Bernd G{\"a}rtner and Ahad~N Zehmakan.
\newblock Majority model on random regular graphs.
\newblock In {\em Latin American Symposium on Theoretical Informatics}, pages
  572--583. Springer, 2018.

\bibitem{ginosar2000}
Yuval Ginosar and Ron Holzman.
\newblock The majority action on infinite graphs: strings and puppets.
\newblock {\em Discrete Mathematics}, 215(1-3):59--71, 2000.

\bibitem{periodic}
Eric Goles and Jorge Olivos.
\newblock Periodic behaviour of generalized threshold functions.
\newblock {\em Discrete mathematics}, 30(2):187--189, 1980.

\bibitem{social1}
Mark Granovetter.
\newblock Threshold models of collective behavior.
\newblock {\em American journal of sociology}, 83(6):1420--1443, 1978.

\bibitem{randomColors}
Yashodhan Kanoria, Andrea Montanari, et~al.
\newblock Majority dynamics on trees and the dynamic cavity method.
\newblock {\em The Annals of Applied Probability}, 21(5):1694--1748, 2011.

\bibitem{mossel2014}
Elchanan Mossel, Joe Neeman, and Omer Tamuz.
\newblock Majority dynamics and aggregation of information in social networks.
\newblock {\em Autonomous Agents and Multi-Agent Systems}, 28(3):408--429,
  2014.

\bibitem{nguyen2020}
Vu~Xuan Nguyen, Gaoxi Xiao, Xin-Jian Xu, Qingchu Wu, and Cheng-Yi Xia.
\newblock Dynamics of opinion formation under majority rules on complex social
  networks.
\newblock {\em Scientific Reports}, 10(1):1--9, 2020.

\bibitem{booleanFunctions}
Ryan O'Donnell.
\newblock {\em Analysis of boolean functions}.
\newblock Cambridge University Press, 2014.

\bibitem{shah2009gossip}
Devavrat Shah.
\newblock {\em Gossip algorithms}.
\newblock Now Publishers Inc, 2009.

\bibitem{berryEsseen}
Irina Shevtsova.
\newblock On the absolute constants in the berry-esseen type inequalities for
  identically distributed summands.
\newblock {\em arXiv preprint arXiv:1111.6554}, 2011.

\bibitem{tamuz2015}
Omer Tamuz and Ran~J Tessler.
\newblock Majority dynamics and the retention of information.
\newblock {\em Israel Journal of Mathematics}, 206(1):483--507, 2015.

\bibitem{tran-vu}
Linh Tran and Van Vu.
\newblock Reaching a consensus on random networks: The power of few.
\newblock {\em arXiv preprint arXiv:1911.10279}, 2019.

\bibitem{zehmakan}
Ahad~N. Zehmakan.
\newblock {Opinion Forming in Erd{\"o}s-R{\'e}nyi Random Graph and Expanders}.
\newblock In Wen-Lian Hsu, Der-Tsai Lee, and Chung-Shou Liao, editors, {\em
  29th International Symposium on Algorithms and Computation (ISAAC 2018)},
  volume 123 of {\em Leibniz International Proceedings in Informatics
  (LIPIcs)}, pages 4:1--4:13, Dagstuhl, Germany, 2018. Schloss
  Dagstuhl--Leibniz-Zentrum fuer Informatik.

\bibitem{zehmakan2018two}
Ahad~N Zehmakan.
\newblock Two phase transitions in two-way bootstrap percolation.
\newblock {\em arXiv preprint arXiv:1809.10764}, 2018.

\end{thebibliography}

\section{Appendix}
\subsection*{Proof of Lemma \ref{lemma:six Fourier facts}}
\begin{proof} We prove each claim in turn.
\paragraph*{Claims (i) and (ii):} These essentially follow immediately from the facts that $Z_v + \mu_v \in \{-1,1\}$, that $Z_v$ has mean 0, and that $Z_v$ depends only on $\Gamma_{v}$.  We'll prove $1 - \mu_v ^2 = 4 \mu + \mathcal{O}(1/\sigma)$, in the case that $v = r \in R_0$ (the case $v \in B_0$ being analogous), for which we have
\begin{eqnarray*}
1 - \mu_r ^2 &=& 1 - \Bigg( 2\mathbb{P}\Big(Bin(|R_0|-1,p) \geq Bin(|B_0|,p) \Big) -1\Bigg)^2\\
&=& 4 \cdot \mathbb{P}\Big(Bin(|R_0|-1,p) \geq Bin(|B_0|,p) \Big) \cdot \mathbb{P}\Big(Bin(|R_0|-1,p) < Bin(|B_0|,p) \Big).
\end{eqnarray*}
And then we simply use the facts\footnote{To prove these, without loss of generality, suppose $y \leq x = \Omega(n)$.  We then condition on the value of $Bin(y,p)$ and use that $\mathbb{P}(Bin(x,p) = t) = \mathcal{O}(1/\sqrt{xp(1-p)}) = \mathcal{O}(1/\sigma)$ uniformly for all $t$.} that if $x+y = \Omega(n)$ then [with $\sigma = \sqrt{np(1-p)}$]
\begin{eqnarray*}
\mathbb{P}\Big(Bin(x,p) \geq Bin(y,p) \Big) &=& \mathbb{P}\Big(Bin(x+1,p) \geq Bin(y,p) \Big) + \mathcal{O}(1/\sigma),\\
\mathbb{P}\Big(Bin(x,p) \geq Bin(y,p) \Big) &=& \mathbb{P}\Big(Bin(x,p) \geq Bin(y+1,p) \Big) + \mathcal{O}(1/\sigma),\\
\mathbb{P}\Big(Bin(x,p) > Bin(y,p) \Big) &=& \mathbb{P}\Big(Bin(x,p) \geq Bin(y,p) \Big) + \mathcal{O}(1/\sigma).
\end{eqnarray*}

\paragraph*{Claims (iii) and (iv):}  For these, suppose $\emptyset \neq S \subseteq \mathcal{E}$.  Write $S$ as $S = S_1 \cup S_2$ where $S_1 = S \setminus (R_0 \times B_0)$ and $S_2 = S \cap (R_0 \times B_0)$ [so $S$ consists of $|S_1|$ edges from $v$ to a vertex of the same color and $|S_2|$ edges from $v$ to the other color].  Define the random variable
\[
W_{v;S} := \begin{cases} Bin(|R_0|-1-|S_1|, p)-Bin(|B_0|-|S_2|,p) \qquad &\text{ if $v \in R_0$}\\Bin(|B_0|-1-|S_1|, p)-Bin(|R_0|-|S_2|,p) \qquad &\text{ if $v \in B_0$}.\end{cases}
\]
Then by conditioning on the edges in $S$, we have (for $S \neq \emptyset$)
\begin{eqnarray*}
\widehat{Z_v}(S) &=& \mathbb{E}\left[Z_v (\vec{x}) \Phi_{S} (\vec{x}) \right]\\
&=& \sum_{I \subseteq S_1} \sum_{J \subseteq S_2} p^{|I|+|J|} (1-p)^{|S|-|I|-|J|} \bigg[2 \mathbb{P}\Big(W_{v;S} \geq |J|-|I|\Big) -1 \bigg]\\
& & \qquad \times \left( \dfrac{2-2p}{2\sqrt{p(1-p)}} \right)^{|I|+|J|} \left(\dfrac{-2p}{2\sqrt{p(1-p)}} \right)^{|S|-|I|-|J|}\\
&=& 2 \left(-\sqrt{p(1-p)} \right)^{|S|} \sum_{I \cup J \subseteq S} (-1)^{|I|+|J|} \mathbb{P}\Big(W_{v;S} \geq |J|-|I|\Big).
\end{eqnarray*}
Since $S \neq \emptyset$, let $e \in S$.  If $e \in S_1$, then by conditioning on whether $e \in I \cup J$, we have
\begin{eqnarray*}
& & \sum_{I \cup J \subseteq S} (-1)^{|I|+|J|} \mathbb{P}\Big(W_{v;S} \geq |J|-|I|\Big)\\
& & \qquad \quad = \sum_{I \cup J \subseteq (S \setminus \{e\})} (-1)^{|I|+|J|} \Bigg[ \mathbb{P}\Big(W_{S} \geq |J|-|I|\Big) - \mathbb{P}\Big(W_{v;S} \geq |J|-|I| - 1 \Big)\Bigg]\\
& & \qquad \quad = -\sum_{I \cup J \subseteq (S \setminus \{e\})} (-1)^{|I|+|J|} \Bigg[ \mathbb{P}\Big(W_{v;S} = |J|-|I|-1\Big) \Bigg].
\end{eqnarray*}
Similarly if $e \in S_2$, then $\displaystyle \widehat{Z_v}(S) = 2(-\sqrt{p(1-p)})^{|S|} \sum_{I \cup J \subseteq (S \setminus \{e \})} (-1)^{|I|+|J|}\mathbb{P}\Big(W_{v;S} = |J|-|I| \Big)$.  In either case, if $|S| \leq 10 k^2$ is bounded, these probabilities are at most $\mathcal{O}(1/\sqrt{np(1-p)})$, establishing (iii).  We prove (iv) by picking some other edge $e' \in S$ and conditioning on whether or not $e' \in I \cup J$ to obtain
\begin{eqnarray*}
|\widehat{Z}_v (S)| &\leq& 2(\sqrt{p(1-p)})^{|S|} \sum_{I \cup J \subseteq (S \setminus \{e, e'\})} \sup_{L} \bigg | \mathbb{P}(W_{v;S} = L+1)-\mathbb{P}(W_{v;S} = L) \bigg|\\
&=& \mathcal{O} \left(\left(\sqrt{p(1-p)}\right)^{|S|}\sup_{L} \bigg | \mathbb{P}(W_{v;S} = L+1)-\mathbb{P}(W_{v;S} = L) \bigg| \right).
\end{eqnarray*}
And this supremum is $\mathcal{O}\left( \frac{1}{np(1-p)} \right)$ by Lemma \ref{lemma:stupid result}.

\paragraph*{Claim (v):} For this, we first note that by claim (ii), this sum reduces to only two terms:
\begin{eqnarray*}
\sum_{e \in E} \Big(\widehat{Z_r} (e) - \widehat{Z_b}(e) \Big) \widehat{Z_v} (e) &=& \widehat{Z_r} (r v) \widehat{Z_v} (r v) - \widehat{Z_b}(b v) \widehat{Z_v} (b v)\\
&=& \widehat{Z_v} (r v) \Big( \widehat{Z_r} (r v) + \widehat{Z_b}(b v) \Big) - \widehat{Z_b}(b v) \Big( \widehat{Z_v} (b v) + \widehat{Z_v} (r v) \Big).
\end{eqnarray*}
Suppose $r' \in R_0$ and $b' \in B_0$ are additional vertices.  From our computation in (iii), we have
\begin{eqnarray*}
\widehat{Z_{r}}(r r') &=& 2 \sqrt{p(1-p)} \cdot \mathbb{P}\Big(Bin(|R_0|-2,p)-Bin(|B_0|,p) = -1\Big),\\
\widehat{Z_{b}}(b b') &=& 2 \sqrt{p(1-p)} \cdot \mathbb{P}\Big(Bin(|B_0|-2,p)-Bin(|R_0|,p) = -1\Big),\\
\widehat{Z_{r}}(r b) = \widehat{Z_{b}}(r b) &=& -2 \sqrt{p(1-p)} \cdot \mathbb{P}\Big(Bin(|R_0|-1,p)-Bin(|B_0|-1,p) = 0\Big).
\end{eqnarray*}
By Lemma \ref{lemma:stupid result}, the above three probabilities are all within $\dfrac{C}{np(1-p)}$ of each other, which implies $\widehat{Z_r}(rv) + \widehat{Z_b}(bv)$ as well as $\widehat{Z_v}(bv) + \widehat{Z_v}(rv)$ are both at most $\mathcal{O}\left( \dfrac{1}{n\sqrt{p(1-p)}}\right)$.  Combining this with claim (iii) then establishes (v).

\paragraph*{Claim (vi):} For this, we use the fact that $Z_v + \mu_v \in \{-1,1\}$ to obtain
\begin{eqnarray*}
Z_v ^L &=& (Z_v + \mu_v - \mu_v) ^L = \sum_{i=0} ^{L} {L \choose i} (Z_v +\mu_v) ^{i} (-\mu_v)^{L-i}\\
&=& Z_v \dfrac{(1-\mu_v)^L - (-1)^L (1+\mu_v)^L}{2} + (1-\mu_v ^2) \dfrac{(1-\mu_v)^{L-1} + (-1)^L (1+\mu_v)^{L-1}}{2}. \qedhere
\end{eqnarray*}
\end{proof}

\subsection*{Proof of Lemma \ref{lemma:sway over some vertices}}
\begin{proof}
We will prove that with high probability $|T_R|$ and $|T_B|$ are both close to their means and that each has expected value of about $\frac{(|V|/2) |X| p}{\sigma \sqrt{2\pi}}$, which will finish the proof since $T_R \cap T_B = \emptyset$.  For this, we'll focus on $T_R$ (the case of $T_B$ being almost identical).  In what follows, let $m = |R| = |B|$.

For $v \in R$, let $T_v$ denote the indicator for the event that $0 < d_{B} (v) - d_{R} (v) \leq d_{X} (v)$.  Then we have $|T_R| = \sum_{v \in R} T_v$.  So the expected value of $|T_R|$ is $m$ times $\mathbb{E}[T_v]$, and moreover
\begin{eqnarray*}
\mathbb{E}[T_v] &=& \mathbb{P}(0 < Bin(m,p) - Bin(m-1,p) \leq Bin(|X|,p))\\
&=& \sum_{k=1} ^{|X|} \sum_{i=1} ^{k} \mathbb{P}(Bin(m,p) - Bin(m-1,p) = i) \mathbb{P}(Bin(|X|,p)=k)
\end{eqnarray*}

Let $W = Bin(m,p)-Bin(m-1,p)$.  For each $i$ we have $\Big| \mathbb{P}(W=i) - \mathbb{P}(W=0) \Big| \leq C i / \sigma^2$, by Lemma \ref{lemma:stupid result}.  Therefore we have
\begin{eqnarray*}
\mathbb{E}[T_v] &=& \sum_{k=1} ^{|X|} \sum_{i=1} ^{k} \left(\mathbb{P}(W=0) + \mathcal{O}(i/\sigma^2) \right) \mathbb{P}(Bin(|X|,p)=k)\\
&=& \mathbb{P}(W=0)\mathbb{E}[Bin(|X|,p))] + \mathcal{O}\left( \dfrac{1}{\sigma^2} \sum_{k=1} ^{|X|} \dfrac{k(k+1)}{2} \mathbb{P}(Bin(|X|,p)=k) \right)\\
&=& \mathbb{P}(W=0)|X|p + \mathcal{O}\left( \dfrac{1}{mp(1-p)} |X|p (1+|X|p) \right) = \mathbb{P}(W=0)|X|p + \mathcal{O}\left( \dfrac{(1+|X|p)^2}{\sigma^2} \right)\\
&=& \dfrac{|X|p}{\sigma \sqrt{2\pi}} + \mathcal{O}\left( \dfrac{(1+|X|p)^2}{\sigma^2} \right).
\end{eqnarray*}
Moreover, for $u \neq v$, the covariance $Cov(T_u, T_v)$ is given by
\begin{eqnarray*}
Cov(T_u, T_v) &=& p(1-p) \Bigg[ \mathbb{E}(T_u | u \sim v) - \mathbb{E}(T_u | u \not \sim v) \Bigg] \Bigg[ \mathbb{E}(T_v | u \sim v) - \mathbb{E}(T_v | u \not \sim v) \Bigg]\\
&=& p(1-p) \Bigg[ \mathbb{P}(Bin(m,p)-Bin(m-2,p) = Bin(|X|,p) + 1)\\
& & \qquad \qquad \qquad \qquad - \mathbb{P}(Bin(m,p)-Bin(m-2,p) = 1) \Bigg]^2\\
&\leq& p(1-p) \mathbb{E}\left[ \left(\dfrac{C Bin(|X|,p)}{mp(1-p)}\right)^2 \right] = \mathcal{O}\left( \dfrac{(1+|X|p)^2}{m \sigma^2} \right),
\end{eqnarray*}
where the last inequality again follows from Lemma \ref{lemma:stupid result}.  Thus, the variance of $|T_R|$ is satisfies
\begin{eqnarray*}
Var(|T_R|) &=& m Var(T_u) + m(m-1) Cov(T_u, T_v) = \mathcal{O}\left( \dfrac{m|X|p}{\sigma} \right) + \mathcal{O}\left( \dfrac{m (1+p|X|)^2}{\sigma^2} \right)\\
&=& \mathcal{O}(m|X|p/\sigma) = \mathcal{O} \left ( |V| p \dfrac{|X|}{\sigma} \right) = o \left ( |V| \cdot \left[\dfrac{|X|p}{\sigma} \right]^{1-2\varepsilon} \right),
\end{eqnarray*}
which holds because of the assumption that $1+|X|p = o\Big(\sigma \Big)$.  Thus, the second-moment method provides the desired concentration statement of $|T_R|$ about its mean.
\end{proof}
\end{document}